\documentclass[11pt]{article}

\usepackage{amsmath}
\usepackage{latexsym}
\usepackage{amsfonts}
\usepackage{amssymb}
\usepackage{amscd}
\usepackage{theorem}
\usepackage{a4wide}

\newtheorem{theorem}{Theorem}
\newtheorem{lemma}[theorem]{Lemma}
\newtheorem{proposition}[theorem]{Proposition}
\newtheorem{corollary}[theorem]{Corollary}

{\theorembodyfont{\rmfamily}

}

\newenvironment{proof}{    
  \noindent
  \textbf{Proof.}}{
  \hfill $\Box$
  \vspace{3mm}
}

\numberwithin{equation}{section}

\newcommand{\N}{\mathbb{N}} 
\newcommand{\R}{\mathbb{R}} 
\newcommand{\C}{\mathbb{C}} 
\newcommand{\D}{\mathbb{D}} 
\newcommand{\T}{\mathbb{T}} 

\begin{document}

\title{The spectrum of composition operators induced by a rotation in the space of all analytic functions on the disc }

\author{Jos\'{e} Bonet}

\date{}

\maketitle

\begin{abstract}
A characterization of those points in the unit disc which belong to the spectrum of a composition operator $C_{\varphi}$, defined by a rotation $\varphi(z)=rz$ with $|r|=1$, on the space $H_0(\D)$ of all analytic functions on the unit disc which vanish at $0$, is given. Examples show that the point $1$ may or may not belong to the spectrum of $C_{\varphi}$, and this is related to Diophantine approximation. Our results complement recent work by Arendt, Celari\`es and Chalendar.
\end{abstract}

\renewcommand{\thefootnote}{}



The aim of this note is to present a few results about the spectrum $\sigma(C_{\varphi_r}, H_0(\D))$ of composition operators, defined by a self map of the unit disc $\D$ given by $\varphi_r(z):=rz$, for a complex number $r$ with $|r|=1$,  when acting on the space $H_0(\D)$ of all analytic functions on the disc $\D$ vanishing at $0$. We prove in Theorem \ref{main} that there exist $r,s \in \C$ with $|r|=|s|=1$, which are not roots of unity, such that the corresponding self-maps $\varphi_r$ and $\varphi_s$ satisfy $1 \notin \sigma(C_{\varphi_r}, H_0(\D))$ and $1 \in \sigma(C_{\varphi_s}, H_0(\D))$. The example of $r$ uses Diophantine irrational numbers, see Proposition \ref{diophantine}; whereas the example of $s$ as indicated requires a construction in Proposition \ref{construction} of irrational numbers $x$ with a sequence of rational numbers very rapidly converging to $x$.

Here the space $H(\D)$ of all analytic functins on $\D$ is endowed with the complete metrizable locally convex topology of uniform convergence on the compact subsets of $\D$. An analytic self-map  $\varphi$ of $\D$ defines a continuous composition operator $C_{\varphi}: H(\D) \rightarrow H(\D), \ C_{\varphi}(f):= f \circ \varphi$ for each $f \in H(\D)$. The monographs of Cowen and   \cite{CM} and Shapiro \cite{Shap} are standard references for composition operators on spaces of analytic functions.

Let $T:X \rightarrow X$ be a continuous linear operator on a Fr\'echet space $X$, i.e.\ a complete metrizable locally convex space. We write $T \in \mathcal{L}(X)$. The \textit{resolvent set} $\rho(T)$ of $T$ consists of all $\lambda\in\C$ such that $R(\lambda,T):=(\lambda I- T)^{-1}$ is a continuous linear operator, that is $T - \lambda I: X \rightarrow X$ is bijective and has a continuous inverse. Here $I$ stands for the identity operator on $X$. The set  $\sigma(T):=\C\setminus \rho(T)$ is called the \textit{spectrum} of $T$. The \textit{point spectrum} $\sigma_{pt}(T)$ of $T$ consists of all $\lambda\in\C$ such that $(\lambda I-T)$ is not injective. If we need to stress the space $X$, then we  write $\sigma(T;X)$, $\sigma_{pt}(T;X)$ and $\rho(T;X)$. Unlike for Banach spaces $X$, it may happen that $\rho(T)=\emptyset$ or that $\rho(T)$ is not open.  This is the reason why some authors (see e.g.\ \cite{V} and \cite{ABR}) prefer to consider  $\rho^*(T)$ instead of $\rho(T)$ consisting of all $\lambda\in\C$ for which  there exists $\delta>0$ such that $B(\lambda,\delta):=\{z\in\C\colon |z -\lambda| < \delta\} \subset \rho(T)$ and $\{R(\mu,T) \colon \ \mu\in  B(\lambda,\delta)  \}$ is equicontinuous in $\mathcal{L}(X)$. The \textit{Waelbroeck spectrum} is defined by $\sigma^*(T):=\C\setminus \rho^*(T)$. It  is a closed set  containing $\sigma(T)$. The reader is referred to the book of Meise and Vogt \cite{Meise_Vogt} for functional analysis and Fr\'echet spaces.

Arendt, Celari\`es and Chalendar \cite[Theorem 4.1]{ACC} prove that if $\varphi: \D \rightarrow \D$ is an analytic map with an interior fixed point $\varphi(a)=a \in \D$ and $0 < |\varphi'(a)| < 1$, then the spectrum of the composition operator $C_{\varphi}$ is
$$
\sigma(C_{\varphi}, H(\D)) = \sigma^*(C_{\varphi}, H(\D)) = \{ 0 \} \cup \{ \varphi'(a)^n \ : \ n=0,1,2,... \}.
$$
Moreover, in case $\varphi'(a)=0$, then \cite[Theorem 4.10]{ACC} implies that $\sigma(C_{\varphi}, H(\D)) = \sigma^*(C_{\varphi}, H(\D)) = \{ 0,1 \}$.

The eigenvalues and spectrum of composition operators on Banach spaces of analytic functions has been investigated by many authors. We refer the reader to \cite{ACC}, \cite{AL} and \cite{Shapart} and the references therein.

If $\varphi$ is an automorphism with an interior fixed point $\varphi(a)=a \in \D$,
then $0 \notin \sigma(C_{\varphi}, H(\D))$, as follows for example from \cite[Corollary 2.4]{ACC}. Moreover, the Moebius transform $\psi_a(z):=(a-z)/(1-\overline{a}z))$ satisfies $\psi_a(a)=0$ and $\psi_a^{-1} = \psi_a$, and
$$
\sigma(C_{\psi_a} C_{\varphi} C_{\psi_a}, H(\D)) = \sigma(C_{\varphi}, H(\D)).
$$
Accordingly to investigate the spectrum in this case it is enough to consider the case of
rotations $\varphi_r(z):=rz, z \in \D,$ for each $r \in \C$ with $|r|=1$.

In the rest of the article we denote by $H_0(\D)$ the space of all the analytic functions on $\D$ such that $f(0)=0$, which is a closed hyperplane of $H(\D)$, hence a Fr\'echet space when endowed with the induced topology. Clearly $C_{\varphi_r}(H_0(\D)) \subset H_0(\D)$ for each $|r|=1$. We give several results about the spectrum of the composition operator $C_{\varphi_r}$ when acting on $H_0(\D)$ which complement the aforementioned  theorems in \cite{ACC}.

Our first result collects several general facts. From now on we denote by $\T$ the unit circle.

\begin{proposition}\label{generalfacts}
Let $\varphi_r$ be the self map of $\D$ given by $\varphi_r(z):=rz, z \in \D$ for some $r \in \T$. Then
\begin{itemize}
\item[(1)] $\sigma_{pt}(C_{\varphi_r}, H(\D)) = \{ r^n \ | \ n=0,1,2,...  \}$.

\item[(2)] $\sigma^*(C_{\varphi_r}, H(\D))$ is contained in the unit circle $\T$. The same holds for $\sigma^*(C_{\varphi_r}, H_0(\D))$

\item[(3)] If $r$ is a root of unity and $m$ is the minimum positive integer such that  $r^m = 1$, then
$$\sigma(C_{\varphi_r}, H(\D)) = \sigma^*(C_{\varphi_r}, H(\D)) = \sigma_{pt}(C_{\varphi_r}, H(\D)) = \{1, r, r^2,...r^{m-1} \}.$$ Moreover ${\rm Ker}(C_{\varphi_r} - r^j I)$ is infinite dimensional for each $j=0,...,m-1$.
The same result holds for $H_0(\D)$ instead of $H(\D)$.

\item[(4)] If $r$ is not a root of unity, then $\sigma_{pt}(C_{\varphi_r}, H_0(\D)) = \{ r^n \ | \ n=1,2,...  \}$, and $\overline{\sigma(C_{\varphi_r}, H_0(\D))} = \sigma^*(C_{\varphi_r}, H_0(\D)) = \T$. The last equality of sets also holds for $H(\D)$.

\end{itemize}
\end{proposition}
\begin{proof}
(1) Clearly $C_{\varphi_r}(z^n)=r^n z^n$ for each $n=0,1,2,...$. On the other hand, if $f(z)=\sum_{n=0}^{\infty} a_n z^n \in H(\D), \ f \neq 0,$ satisfies $f(rz)= \mu f(z)$ for each $z \in \D$, we have $a_n(r^n - \mu) = 0$ for each $n =0,1,2,...$. Since $f \neq 0$, this implies that $\mu = r^n$ for some $n$.

(2)  We first show that $C_{\varphi_r} - \eta I: H(\D) \rightarrow H(\D)$ is bijective for each $\eta \notin \T$, hence an isomorphism by the closed graph theorem. By part (1), $C_{\varphi_r} - \eta I$ is injective. We show that it is also surjective. Given $g(z)=\sum_{n=0}^{\infty} a_n z^n \in H(\D)$, define $f(z):=\sum_{n=0}^{\infty} \frac{a_n}{r^n - \eta} z^n$. There is $\delta >0$ such that $|\eta - r^n| > \delta$ for each $n=0,1,2,..$. Since the series $\sum_{n=0}^{\infty} |a_n| |z|^n$ converges for each $|z| < 1$, we get that $\sum_{n=0}^{\infty} \frac{|a_n|}{|r^n - \eta|} |z|^n$ also converges for each $|z| < 1$, and $f(z)$ is an analytic function on $\D$ such that $(C_{\varphi_r} - \eta I) f = g$.

In particular we have shown that $$(C_{\varphi_r} - \eta I)^{-1}\left( \sum_{n=0}^{\infty} a_n z^n \right) = \sum_{n=0}^{\infty} \frac{a_n}{r^n - \eta} z^n.$$ for each $\eta \notin \T$ and each $\sum_{n=0}^{\infty} a_n z^n \in H(\D)$.

Now fix $\mu \notin \T$ and select $\delta >0$ such that $|\eta - r| > \delta$ for each $r \in \T$ and each $\eta \in \C$ with $|\eta - \mu| < \delta$. To prove that $\mu \notin \sigma^*(C_{\varphi_r}, H(\D))$, by the uniform boundedness principle,  it is enough to show that the set
$$
B(g):= \{ (C_{\varphi_r} - \eta I)^{-1} g \ | \ |\eta - \mu| < \delta \}
$$
is bounded in $H(\D)$ for each $g = \sum_{n=0}^{\infty} a_n z^n \in H(\D)$. To see this, fix $\alpha \in ]0,1[$. We have
$$
\sup_{|z| \leq \alpha} |(C_{\varphi_r} - \eta I)^{-1} g(z)| = \sup_{|z| \leq \alpha} \left| \sum_{n=0}^{\infty} \frac{a_n}{r^n - \eta} z^n   \right| \leq \delta^{-1} \sum_{n=0}^{\infty} |a_n| \alpha^n.
$$

The proof that $\sigma^*(C_{\varphi_r}, H_0(\D)) \subset \T$ is the same.

(3) Clearly
$$
\{1, r, r^2,...r^{m-1} \} = \sigma_{pt}(C_{\varphi_r}, H(\D)) \subset \sigma(C_{\varphi_r}, H(\D)) \subset \sigma^*(C_{\varphi_r}, H(\D)).
$$
The proof each $\mu \notin \{1, r, r^2,...r^{m-1} \}$ does not belong to $\sigma^*(C_{\varphi_r}, H(\D))$ is similar to the one in part (2), just taking $\delta>0$ such that $|\eta - r^n| > \delta$ for each $n=0,...,m-1$ and each $\eta \in \C$ with $|\eta - \mu| < \delta$.

The space ${\rm Ker}(C_{\varphi_r} - r^j I), j=0,...,m-1,$ of eigenvectors is infinite dimensional because $C_{\varphi_r}(z^{km + j}) = r^j  z^{km + j}$ for each $k \in \N$.

In the case of $H_0(\D)$, observe that, although the constants do not belong to this space, we have $C_{\varphi_r}(z^{km}) = z^{km}$ for each $k \in \N$. The rest of the proof is similar.

(4) If $r$ is not a root of unity, then $r^n \neq 1$ for each $n \in \N$. Since the constants do not belong to $H_0(\D)$, the argument in part (1) yields $\sigma_{pt}(C_{\varphi_r}, H_0(\D)) = \{ r^n \ | \ n \in \N \}$.

Therefore, by part  (2),
$$
\{ r^n \ | \ n \in \N \} \subset \sigma(C_{\varphi_r}, H_0(\D)) \subset \sigma^*(C_{\varphi_r}, H_0(\D)) \subset \T.
$$
By Kronecker's Theorem (see \cite[Theorem 2.2.4]{Queffelec_book} or \cite[Theorem 438]{HW}), $\{ r^n \ | \ n \in \N \}$ is dense in $\T$. This implies
$$
\overline{\sigma(C_{\varphi_r}, H_0(\D))} = \sigma^*(C_{\varphi_r}, H_0(\D)) = \T.
$$
The proof that $\sigma^*(C_{\varphi_r}, H(\D)) = \T$ is obtained with the same argument.
\end{proof}

There are examples of continuous linear operators $T$ on a Fr\'echet space $E$ such that $\overline{\sigma(T,E)}$ is properly contained in $\sigma^*(T,E)$. See Remark 3.5 (vi) in \cite{ABR}. Compare with the statement in Proposition \ref{generalfacts} (4).

\begin{proposition}\label{characterization}
Let $r \in \T$  satisfy that $r^n \neq 1$ for each $n \in \N$.
Let $\varphi_r(z):=rz, z \in \D$.

The following conditions are equivalent for $\lambda \in \T $ with $\lambda \neq r^n$ for each $n \in \N$:
\begin{itemize}

\item[(i)] $\lambda \notin \sigma(C_{\varphi_r}, H_0(\D))$.

\item[(ii)] $\forall \alpha \in ]0,1[ \ \ \ \exists \beta \in ]\alpha, 1[ \ \ \ {\rm such \ that} \ \ \
\sup_{n \in \N} \frac{1}{|r^n - \lambda|} \big(\frac{\alpha}{\beta}   \big)^n < \infty. $
\end{itemize}
\end{proposition}
\begin{proof}
By the closed graph theorem, condition (i) holds if and only if $C_{\varphi_r} - \lambda I: H_0(\D) \rightarrow H_0(\D)$ is bijective. Since $\lambda \neq r^n$ for each $n \in \N$, $C_{\varphi_r} - \lambda I$ is injective by Proposition \ref{generalfacts}.

(ii) implies (i). It is enough to show that condition (ii) implies  that $C_{\varphi_r} - \lambda I: H_0(\D) \rightarrow H_0(\D)$ is surjective. Fix $g(z)=\sum_{n=1}^{\infty} a_n z^n \in H_0(\D)$ and define $f(z):=\sum_{n=1}^{\infty} \frac{a_n}{r^n - \lambda} z^n$. Given $\alpha \in ]0,1[$, select $\beta \in ]\alpha, 1[$ as in condition (ii). Hence, for each $|z| \leq \alpha$, we have
$$
\sum_{n=1}^{\infty} \frac{|a_n|}{|r^n - \lambda|} |z|^n   \leq  C \sum_{n=1}^{\infty} |a_n| \beta^n.
$$
Therefore $f \in H_0(\D)$ and, clearly, $(C_{\varphi_r} - \lambda I) f = g$.

(i) implies (ii). Condition (i) implies that the operator $T:=(C_{\varphi_r} - \lambda I)^{-1}: H_0(\D) \rightarrow H_0(\D)$ given by $T(\sum_{n=1}^{\infty} a_n z^n):= \sum_{n=1}^{\infty} \frac{a_n}{r^n - \lambda} z^n$ is continuous. It is well known that a fundamental system of seminorms of the Fr\'echet topology of $H_0(\D)$ is given by $$P_{\alpha}(\sum_{n=1}^{\infty} a_n z^n):= \sup_{n \in \N} |a_n| \alpha^n \ \ \ \ \alpha \in ]0,1[.$$ Therefore the continuity of the operator $T$ implies that for each $\alpha \in ]0,1[$, one can select $\beta \in ]\alpha, 1[$ and $C>0$
such that $$P_{\alpha}(\sum_{n=1}^{\infty} \frac{a_n}{r^n - \lambda} z^n) \leq C P_{\beta}(\sum_{n=1}^{\infty} a_n z^n)$$
for each $\sum_{n=1}^{\infty} a_n z^n \in H_0(\D)$.
Evaluating in each monomial $z^j, j \in \N,$ we get $\frac{1}{|r^j - \lambda|} \alpha^j \leq C \beta^j$. Hence $\sup_{j \in \N} \frac{1}{|r^j - \lambda|} \big(\frac{\alpha}{\beta}   \big)^j < \infty$, which is condition (ii).
\end{proof}

A real number $x \in \R$ is called \textit{Diophantine} if there are $\delta \geq 1$ and $c(x)>0$ such that
$$
\left|x - \frac{p}{q}\right| \geq \frac{c(x)}{q^{1+\delta}}
$$
for all rational numbers $p/q$. See \cite[Definition 3.1.10]{Queffelec_book}. The complement of the set of Diophantine numbers has Lebesgue measure $0$; see \cite[page 43]{CG}. By a Theorem of Liouville (see \cite[Theorem 191]{HW} or \cite[Proposition 3.1.3]{Queffelec_book}) every irrational algebraic number of degree larger or equal 2 is Diophantine.

\begin{corollary}\label{diophantine}
If $x \in \R$ be a Diophantine number, then the rotation $\varphi_r(z)=rz, z \in \D$ for  $r:=e^{i 2 \pi x}$ satisfies that $1 \notin \sigma(C_{\varphi_r}, H_0(\D))$. In particular $1 \in \sigma^*(C_{\varphi_r}, H_0(\D)) \setminus \sigma(C_{\varphi_r}, H_0(\D))$
\end{corollary}
\begin{proof}
If  $x$ is a Diophantine number, then there are $d>0$ and $\delta \geq 1$ such that $r:=e^{i 2 \pi x}$ satisfies
$|r^n - 1| \geq d/n^{\delta}$ for all $n \in \N$, \cite[page 43]{CG}. Hence
there is $D>0$ such that $1/|r^n - 1| \leq D n^{\delta}$ for each $n \in \N$.
Given $\alpha \in ]0,1[$ take any $\beta \in ]\alpha, 1[$. We have, for each $n \in \N$,
$$
\frac{1}{|r^n - \lambda|} \left(\frac{\alpha}{\beta} \right)^n \leq D \sup_{n \in \N} n^{\delta} \left(\frac{\alpha}{\beta} \right)^n.
$$
This implies $\sup_{n \in \N} \frac{1}{|r^n - \lambda|} \big(\frac{\alpha}{\beta}   \big)^n < \infty$, and the conclusion follows from Proposition \ref{characterization}.

\end{proof}

\begin{lemma}\label{lemma1}
Let $x \in \R$ be an irrational number such that there are $\alpha \in ]0,1[$ and a sequence of rational numbers $(p_j/q_j)_{j=1}^{\infty}$ with $\lim_{j \rightarrow \infty} q_j = + \infty$ such that
$$
\left| x - \frac{p_j}{q_j}  \right| \leq \alpha^{q_j}
$$
for each $j \in \N$. Then the rotation $\varphi_r(z):=rz, z \in \D,$ for $r:=e^{i 2 \pi x}$ satisfies $1 \in \sigma(C_{\varphi_r}, H_0(\D))$.
\end{lemma}
\begin{proof}
The complex number $r:=e^{i 2 \pi x}$ satisfies, for each $j \in \N$,
$$
|r^{q_j} -1| = |e^{i 2 \pi q_j x} - e^{i 2 \pi p_j}| \leq 2 \pi |q_j x - p_j| \leq 2 \pi q_j \alpha^{q_j}.
$$
Proceeding by contradiction, assume that $1 \notin \sigma(C_{\varphi_r}, H_0(\D))$. Then $C_{\varphi_r} -  I: H_0(\D) \rightarrow H_0(\D)$ is surjective. Given $\sum_{n=1}^{\infty} z^n \in H_0(\D)$, its pre-image
$\sum_{n=1}^{\infty} \frac{1}{r^n - 1} z^n$ belongs to $H_0(\D)$. In particular, if we select any $\beta \in ]\alpha,1[$, there is $M>0$ such that $\frac{\beta^n}{|r^n - 1 |} \leq M$ for each $n \in \N$. This yields, for each $j \in \N$,
$$
\beta^{q_j} \leq M|r^{q_j} -1| \leq 2 \pi M q_j \alpha^{q_j}.
$$
Hence $(\beta/\alpha)^{q_j} \leq 2 \pi M q_j$ for each $j \in \N$. This is impossible, since $1 < \beta/\alpha$ and  $q_j$ tends to infinity.
\end{proof}

Here is the construction of irrational numbers $x$ satisfying the assumptions of Lemma \ref{lemma1}, that is, with very fast Diophantine approximation.

\begin{proposition}\label{construction}
For each $m \in \N, m \geq 2,$ there is a positive real number $x(m)$ and  sequences of positive numbers $(p_j)_{j=1}^{\infty} $ and $(q_j)_{j=1}^{\infty}$ with $\lim_{j \rightarrow \infty} q_j = + \infty$ such that
$$
\left| x(m) - \frac{p_j}{q_j}  \right| \leq \big( \frac{1}{m} \big)^{q_j}
$$
\end{proposition}
\begin{proof}
Given $m \in \N, m \geq 2,$ define inductively $q_1:=m$ and $q_{j+1}:= (q_j)^{q_j}$ for $j \in \N$. Clearly $q_j \geq m^j$ and $q_j$ is a divisor of $q_{j+1}$ for each $j \in \N$.
We set $x(m):= \sum_{j=1}^{\infty} 1/q_j$. The series converges, and each partial sum $\sum_{j=1}^{k} 1/q_j$ is a rational number which can be written of the form $p_k/q_k$ for some $p_k \in \N$. Moreover, we have
$$
\left|x(m) - \frac{p_k}{q_k}\right| = \sum_{j=k+1}^{\infty} \frac{1}{q_j} \leq \ \frac{1}{q_{k+1}} \sum_{j=0}^{\infty} \left(\frac{1}{m}\right)^j \leq \frac{2}{q_{k+1}}.
$$
Since $2 m^{q_k} \leq q_{k+1}$ for each $k \in \N$, we have $2/q_{k+1} \leq (1/m)^{q_k}$ for each $k \in \N$.Therefore the rational numbers $p_k/q_k, k \in \N$ satisfy
$$
\left| x(m) - \frac{p_k}{q_k}  \right| \leq \left( \frac{1}{m} \right)^{q_k}
$$
for all $k \in \N$.
\end{proof}

\begin{theorem}\label{main}
There exist $r,s \in \T$ such that the corresponding rotations $\varphi_r$ and $\varphi_s$ satisfy $1 \notin \sigma(C_{\varphi_r}, H_0(\D))$ and $1 \in \sigma(C_{\varphi_s}, H_0(\D))$.
\end{theorem}
\begin{proof}
Take $r=e^{i 2 \pi x}$ for a Diophantine number $x \in \R$ and apply  Corollary \ref{diophantine} to conclude that $1 \notin \sigma(C_{\varphi_r}, H_0(\D))$. On the other hand, by  Proposition \ref{construction} and Lemma \ref{lemma1} we find numbers $s \in \T$ with $1 \in \sigma(C_{\varphi_s}, H_0(\D))$.
\end{proof}

\textbf{Acknowledgement.} The research of this paper was partially
supported by the projects MTM2016-76647-P and GV Prometeo/2017/102.




\noindent \textbf{Author's address:}%
\vspace{\baselineskip}%

Instituto Universitario de Matem\'{a}tica Pura y Aplicada IUMPA,
Universitat Polit\`{e}cnica de Val\`{e}ncia,  E-46071 Valencia, SPAIN

email:jbonet@mat.upv.es

\end{document}